\documentclass{amsart}
\usepackage{graphicx,latexsym,bm,amsmath,amssymb,verbatim,multicol,lscape}
\newtheorem{thm}{Theorem} [section]
\newtheorem{cor}[thm]{Corollary}
\newtheorem{lem}[thm]{Lemma}

\theoremstyle{definition}
\newtheorem{defn}[thm]{Definition}
\theoremstyle{remark}
\newtheorem{rem}[thm]{Remark}
\newtheorem{con}[thm]{Conjecture}
\numberwithin{equation}{section}
\begin{document}
\title{Notes On a Borwein and Choi's conjecture of cyclotomic polynomials with coefficients $\pm1$}%
\author{Shaofang Hong \qquad and \qquad Wei Cao}%
\address{Mathematical College Sichuan University Chengdu, Sichuan 610064 P.R.China}%
\email{hongsf@263.net\qquad caowei433100@vip.sina.com}%
%\thanks{The authors are very grateful to Professor Shaofang Hong for his help and encouragement in the writing of this paper}%
%\subjclass{Primary 11T22,11R18}%
%\date{16/8/2004}%
%\dedicatory{}%
%\commby{}%
% ----------------------------------------------------------------
\begin{abstract}
Borwein and Choi  conjectured that a polynomial $P(x)$ with coefficients $\pm1$ of degree $N-1$ is cyclotomic iff
$$P(x)=\pm \Phi_{p_1}(\pm x)\Phi_{p_2}(\pm x^{p_1})\cdots \Phi_{p_r}(\pm x^{p_1p_2\cdots p_{r-1}})$$ where $N=p_1p_2\cdots p_{r}$ and
the $p_i$ are primes, not necessarily distinct. Here $\Phi_p(x):=(x^p-1)/(x-1)$ is the $p-$th cyclotomic
polynomial. In \cite{1}, they also proved the conjecture for $N$ odd or a power of 2. In this paper we introduce a
so-called $E-$transformation, by which we prove the conjecture for a wider variety of cases and present the key as
well as a new approach to investigate the conjecture.
\end{abstract} \maketitle
% ----------------------------------------------------------------
\section{Introduction}
For a polynomial $p(z)\in \mathbb{C}(z)$ and a positive $\alpha$, define the $L_\alpha$ norm of $p(z)$ as:
$$\|p\|_\alpha:=(\frac{1}{2\pi}\int_0^{2\pi}|p(e^{i\theta})|^\alpha\textup{d}\theta)^{1/\alpha}.$$\par
The polynomials with coefficients $\pm1$ are called Littlewood polynomials by Borwein and Choi in \cite{1}, since
Littlewood \cite{3} raised a number of questions concerning such set of polynomials. The $L_2$ norm of a
Littlewood polynomial with degree $n$ is equal to $\sqrt{n+1}$. One of the older of Littlewood's questions, which
is over fifty years and still remains unsolved, is the following conjecture:
\begin{con}
(Littlewood) There exist two positive constants $c_1$ and $c_2$ such that for any $n$ we can find  a Littlewood
polynomial $p_n$ with degree $n$ satisfying that $c_1\sqrt{n+1}\leq |p_n(z)| \leq c_2\sqrt{n+1}$ for all complex
$z$ with $|z|=1.$
\end{con}
Many of the questions raised concern comparing the behavior of Littlewood polynomials in other norms to the $L_2$
norm, among which is the problem of minimizing the $L_4$ norm. In particular, can Littlewood polynomials of degree
$n$ have $L_4$ norm asymptotically close to $\sqrt{n+1}$?\par For a polynomial
$p(z)=a(z-\alpha_1)(z-\alpha_2)\cdots (z-\alpha_n)\in \mathbb{C}(z)$, its Mahler measure is defined as
$M(p)=|a|\prod_{\alpha_i\geq 1}|\alpha_i|$. Since
$$M(p)=\lim_{\alpha\rightarrow0}\|p\|_\alpha=\exp(\frac{1}{2\pi}\int_0^{2\pi}\log(|p(e^{i\theta})|)\textup{d}\theta)=:\|p\|_0$$
one would expect it to be closely related to the minimization problem for the $L_4$ norm above. The minimum
possible Mahler measure for a Littlewood polynomial is 1 and this is achieved by any monic polynomial in
$\mathbb{Z}(x)$ with all roots of modulus 1, which is called the cyclotomic polynomial.\par To characterize the
cyclotomic Littlewood polynomials, Borwein and Choi \cite{1} raised the following conjecture :
\begin{con}(Borwein and Choi)
A polynomial $P(x)$ with coefficients $\pm1$ of degree $N-1$ is cyclotomic iff
$$P(x)=\pm \Phi_{p_1}(\pm x)\Phi_{p_2}(\pm x^{p_1})\cdots \Phi_{p_r}(\pm x^{p_1p_2\cdots p_{r-1}})\leqno (1.1)$$ where $N=p_1p_2\cdots p_{r}$ and
the $p_i$ are primes, not necessarily distinct, and where $\Phi_p(x):=(x^p-1)/(x-1)$ is the $p-$th cyclotomic
polynomial.
\end{con}
They \cite{1} proved two special cases when $N$ is odd or a power of 2. As an application of Conjecture 1.2,
Borwein, Choi and Ferguson \cite{2} proved that
\begin{thm} \textup{(} \cite{2} \textup{)}
If $$P(x)=\pm \Phi_{p_1}(\pm x)\Phi_{p_2}(\pm x^{p_1})\cdots \Phi_{p_r}(\pm x^{p_1p_2\cdots p_{r-1}})$$ where
$N=p_1p_2\cdots p_{r}$ and the $p_i$ are primes, then
\begin{eqnarray*}
\frac{\|P\|_4^4}{N^2} & \geq & \frac{\|\Phi_2(-x)\Phi_2(-x^2)\cdots\Phi_2(-x^{2^{r-1}})\|_4^4}{4^r}\\
                      &   =  &
                      \frac{(\frac{1}{2}+\frac{5}{34}\sqrt{17})(1+\sqrt{17})^r-(-\frac{1}{2}+\frac{5}{34}\sqrt{17})(1-\sqrt{17})^r}{4^r}.
\end{eqnarray*}
\end{thm}
This paper addresses the investigation of Conjecture 1.2. It presents a new approach that we call the
$E-$transformation. And by this approach, we prove that Conjecture 1.2 is true for a wider variety of cases and
give the key and the direction to further investigate the conjecture. The paper is organized as follows: Section 2
is preparations, including the notations that will be used and simple discussion on Conjecture 1.2. In Section 3,
we introduce the $E-$transformation and offer more cases for which Conjecture 1.2 is true. A concrete example is
taken in Section 4, and through observation and analysis we point out the direction under which Conjecture 1.2 may
be completely solved.

\section{Preparations}
Throughout the paper we always let $2\leq N=2^tM$ with $M$ odd and $2\leq i\leq N-1$. And for convenience, we
define some sets of polynomials in $\mathbb{Z}[x]$ as follows:
\begin{enumerate}
\item[ ]$C(N):=\{P(x)\in \mathbb{Z}[x]\mid \deg P=N-1\text{ and }P(x) \text{ is cyclotomic} \};$
\item[ ]$OC(N):=\{P(x)\in C(N)\mid \text{The coefficients of P(x) are odd} \};$
\item[ ]$LC(N):=\{P(x)\in C(N)\mid \text{The coefficients of P(x) are }\pm1 \};$
\item[ ]$LC(N,i):=\{\sum\nolimits_{n=1}^{N-1}a_nx^n\in LC(N)\mid a_0=a_1=\ldots=a_{i-1}=1,a_i=-1\}.$
\end{enumerate}
Clearly, $C(N)\supset OC(N)\supset LC(N) \supset LC(N,i)$. For every $N\geq2$, there is a special polynomial in
$LC(N): \mathbb{P}_N(x):=1+x+x^2+\cdots+x^{N-1}$. Obviously, Conjecture 1.2 is true for it since $
\mathbb{P}_N(x)=\Phi_{p_1}(x)\Phi_{p_2}(x^{p_1})\cdots \Phi_{p_r}(x^{p_1p_2\cdots p_{r-1}})$ where $N=p_1p_2\cdots
p_{r}$ with $p_i$ prime. The importance of $\mathbb{P}_N(x)$ lies in that all $P(x)\in LC(N)$ can be transformed
from $\mathbb{P}_N(x)$ through a so-called $E-$transformation, which will be proved in Section 3. For any
$\mathbb{P}_N(x)\neq P(x)=\sum_{n=1}^{N-1}a_nx^n\in LC(N)$, it is easy to show that there exists one and only one
polynomial in $\{P(x),P(-x),-P(x),-P(-x)\}$ that belongs to $LC(N,i)$ for some $2\leq i\leq N-1$.\par As usual,
for $z\in \mathbb{Z}^+$ and a prime $p$, let $v_p(z)$ denote the $p-$adic valuation of $z$, i.e. $p^{v_{p(z)}}|z$
but $p^{v_{p(z)}+1}\nmid z$. This notation is also valid for the ring $\mathbb{Z}(x)$.\par Suppose
$P(x)=\sum_{n=1}^{N-1}a_nx^n\in LC(N)$. Since $P(x)$ is cyclotomic, it can be written as the product of the
irreducible $d-$th cyclotomic polynomials $\Phi_d(x)$ where $d\geq1$, i.e. $P(x)=\prod_{d\geq1}\Phi_d^{e(d)}(x)$
where $e(d)=v_{\Phi_d(x)}(P(x))$. Borwein and Choi [1] have further proved that
$P(x)=\prod_{d|2N}\Phi_d^{e(d)}(x)$. Suppose all roots of $P(x)$ are $x_1,x_2,\ldots,x_{N-1}$, then we have
$P(x)=\prod_{n=1}^{N-1}(x-x_n)$. So there are three expressions of $P(x)$ and we will choose its suitable
expression according to circumstances.
\par Let $S_k(P)$ be the sum of the $k-$th powers of all the roots of
$P(x)$, i.e. $S_k(P)=\sum\nolimits_{n=1}^{N-1} x_n^k$ and $C_d(k)$ be the sum of the $k-$th powers of the
primitive $d-$th roots of unity (Ramanujan sum), i.e. $C_d(k)=\sum e^{2\pi i a k/d}$ where $a$ is over a
irreducible set of $d$. If no confusion, we simply denote $S_k(P)$ by $S_k$. It is easy to see that
$S_k=\sum_{d|2N}e(d)C_d(k)$.\par In what follows, when we use the notation $S_k$ it means $1\leq k\leq N-2$.
\begin{lem}
The Ramanujan sum $C_d(k)$ has the following properties:
\begin{enumerate}
\item[(a)]If $(d_1,d_2)=1$, then $C_{d_1d_2}(k)=C_{d_1}(k)C_{d_2}(k)$, i.e. $C_d(k)$ is a multiplicative
arithmetic function with respect to $d$.
\item[(b)]If $(d,k')=1$, then $C_{d}(kk')=C_{d}(k)$.
\item[(c)]Let $p$ be a prime and $n\geq1$, then
$$C_{p^n}(k)=\left\{
\begin{array}{ll} p^n-p^{n-1} & if\quad v_p(k)\geq n\\
-p^{n-1} & if\quad v_p(k)=n-1\\
0        & otherwise
\end{array}\right.$$
\item[(d)]Let $n\geq0$. Define $T_n(k):=C_1(k)+C_2(k)+\cdots+C_{2^n}(k)-C_{2^{n+1}}(k)$, then
$$T_n(k)=\left\{\begin{array}{ll}
2^{n+1} & if\quad v_2(k)=n\\ 0 & else\end{array}\right.$$
\end{enumerate}
\end{lem}
\begin{proof}
(a),(b) and (c) are trivial by the definition of $C_d(k)$.\par For (d), since $\sum_{j=0}^nC_{2^n}(k)$ is the sum
of $k-$th powers of the roots of $\prod_{d|2^n}=x^{2^n}-1$, which equals $2^n$ for $2^n|k$ and zero else, by (b)
and (c) the results follows.
\end{proof}
Let $P(x)=\sum_{n=1}^{N-1}a_nx^n\in LC(N)$, we can characterize $P(x)$ from two aspects: its coefficients, which
are all $\pm1$, and its roots, which are all the primitive roots of unity. What is the relationship between both?
The famous Newton's formula can partially answer this question. Since $P(x)$ is cyclotomic, we have
$x^{N-1}P(1/x)=\pm P(x)$. Thus it follows from Newton's formula that (see \cite{1})
$$S_k+a_1S_{k-1}+\cdots+a_{k-1}S_1+ka_k=0\leqno (2.1)$$ Without loss of generality, suppose
 $P(x)=\sum_{n=1}^{N-1}a_nx^n\in LC(N,i)$, then
we have $a_0=a_1=\ldots=a_{i-1}=1$ and $a_i=-1$. Using (2.1) repeatedly, we will get $S_0=S_1=\ldots=S_{i-1}=-1$
and $S_i=2i-1$. Since $i\geq2$, $S_i$ is the first sum that is not equal to -1. In other words, we have
$i=min\{k:S_k\neq-1\}$. \par However, only Newton's formula is not enough since it is true for all Littlwood
polynomials. To further reveal the relationship between coefficients and roots of $P(x)$, we should sufficiently
utilize its cyclotomic characteristic.
\section{$E-$transformation and some special cases}
\begin{lem}\textup{(} \cite{1} \textup{)}
$P(x)\in OC(N)$ iff
$$P(x)=\prod_{d|M}\Phi_d^{e(d)}(x)\Phi_{2d}^{e(2d)}(x)\cdots \Phi_{2^{t+1}d}^{e(2^{t+1}d)}(x)$$
where $$e(d)+\sum\limits_{n=1}^{t+1}2^{n-1}e(2^nd)=\left\{\begin{array}{ll}2^t & if \quad d|M\quad  and \quad d>1\\
2^t-1 & if\quad d=1 \end{array}\right..$$
\end{lem}
From Lemma 3.1 above, it follows that $e(2^{t+1})=0$.
\par Let $P(x)\in C(N)$. If $e(d)=v_{\Phi_d(x)}(P(x))\leq1$ for any $d\geq1$, we say that $P(x)$ is square-free.
Suppose that $P(x)\in OC(N)$ is square-free, it is easy to show that for $1<d|M$ there are only two cases:
$$(e(d),e(2d),\ldots,e(2^{t}d),e(2^{t+1}d))=(1,1,\ldots,1,0)\text{\quad or \quad  }(0,0,\ldots,0,1). \leqno (3.1)$$
And for $d=1$, there are also only two cases:
$$(e(1),e(2),e(4),\ldots,e(2^{t}))=(0,1,1,\ldots,1)\text{\quad or \quad  }(1,0,1,\ldots,1). \leqno (3.2)$$
\par Now we define an important as well as interesting tranformation $E: OC(N)\rightarrow OC(N)$. Suppose
$P(x)=\prod\nolimits_{d|2N}\Phi_d^{e(d)}(x)\in OC(N)$. For any $0\leq t'\leq t$ and $d'|M$, if
$e(d')\pm1,e(2d')\pm1,\ldots,e(2^{t'}d')\pm1,e(2^{t'+1}d')\mp1\geq0$, define:
\begin{eqnarray*}
E(P|t',d') & = & \Phi_{d'}^{e(d')\pm1}(x)\Phi_{2d'}^{e(2d')\pm1}(x)
                \cdots\Phi_{2^{t'}d'}^{e(2^{t'}d')\pm1}(x)\Phi_{2^{t'+1}d'}^{e(2^{t'+1}d')\mp1}(x)\\
           &   &(\prod_{d|2N,d\neq d'}\Phi_d^{e(d)}(x))
\end{eqnarray*}
Since $\sum\nolimits_{n=0}^{t'-1}2^t=2^{t'}-1$, it follows that $E(P|t',d') \in OC(N)$ by Lemma 3.1. How does
$S_k(P)$ chang during the transformation? By Lemma 2.1 (a) and (d) we have \setcounter{equation}{2}
\begin{eqnarray}
S_k(E(P|t',d')) & = & S_k(P)+(\pm C_{d'}(k)\pm C_{2d'}(k)\pm\cdots\pm C_{2^{t'}d'}(k)\mp C_{2^{t'+1}d'}(k)) \nonumber\\
          & = & S_k(P)\pm(C_{1}(k)+C_{2}(k)+\cdots+C_{2^{t'}}(k)-C_{2^{t'+1}}(k)) C_{d'}(k) \nonumber\\
          & = & S_k(P)\pm T_{t'}(k)C_{d'}(k) \nonumber\\
          & = &  \begin{cases} S_k(P)\pm 2^{t'+1} C_{d'}(k) & if \quad v_2(k)=t'\\ S_k(P) & else \end{cases}
\end{eqnarray}
We call $(t',d')$ the parameters set of $E$. For convenience, we usually omit the parameters and simply denote
$E(P|t',d')$ by $E(P)$. Moreover, if we make other $E-$transformation of $E(P|t',d')$, e.g.
$E(E(P|t',d')|t'',d'')$, we still denote it by $E(P)$. For finite $E-$transformations of $P(x)$, let $G(E)$ denote
the set of all its parameters. To be precise, suppose that $n$ $E-$transformations have been made in all and
$(t_i,d_i)$ is the parameters set for $i-$th $E-$transformation where $1\leq i\leq n$, then we have
$G(E)=\{(t_i,d_i)\mid 1\leq i\leq n)\}$.
\par We will show that any $P(x)\in OC(N)$ can be transformed into $\mathbb{P}_N(x)$ through finite $E-$transformations,
i.e.,
$$E(P)=\mathbb{P}_N(x)=\frac{x^N-1}{x-1}=\prod_{1<d|N}\Phi_{d}(x).$$
It is equivalent to the following:
\begin{lem}
Suppose that $w=(w_0,w_1,\ldots, w_t,w_{t+1})$ with $w_n\geq0$ for $0\leq n\leq t+1$ satisfies that
$w_0+\sum_{n=1}^{t+1}2^{n-1}w_n=2^t\; (or:\;w_0+\sum_{n=1}^{t+1}2^{n-1}w_n=2^t-1)$. For any  $0\leq t'\leq t$, if
$w_0\pm1,w_1\pm1,\ldots, w_{t'}\pm1,w_{t'+1}\mp1\geq0$, define an $E-$ transformation of $w$ as
$E(w)=(w_0\pm1,w_1\pm1,\ldots, w_{t'}\pm1,w_{t'+1}\mp1,w_{t'+2},\ldots,w_{t+1})$. Then by finite
$E-$transformations on $w$, we can get $$E(w)=(1,1,\ldots,1,0)\;
(or:\;E(w)=(0,1,\ldots,1,0)).$$\end{lem}\begin{proof} Firstly, we define two functions of $w$:
$$r:=\left\{\begin{array}{cc} min\{n:w_n=0,0\leq n\leq t+1\} & if \quad \exists w_n=0\; 0\leq n\leq t+1\\ -1 & else\end{array}\right.;$$
$$R:=\left\{\begin{array}{cc} max\{n:w_n\geq2,0\leq n\leq t+1\} & if \quad \exists w_n\geq2\; 0\leq n\leq t+1\\ -1 & else\end{array}\right..$$
\par Now we can give an algorithm as follows: \par\textcircled{1} Since $w_0+\sum_{n=1}^{t+1}2^{n-1}w_n=2^t$, then $w_0+w_1$ must be even.
Repeat $E-$transformation on $w_0$ and $w_1$ until $w_0=w_1$.\par \quad (or: Since
$w_0+\sum_{n=1}^{t+1}2^{n-1}w_n=2^t-1$, then $w_0+w_1$ must be odd. Repeat $E-$transformation on $w_0$ and $w_1$
until $w_0=w_1-1$.)\par\textcircled{2} Compute $r$ and $R$. If $R=-1$, then go to \textcircled{4}; else continue
next step.\par \par\textcircled{3} a. If $r<R$, then
 $E(w)=(w_0+1,w_1+1,\ldots,w_{R-1}+1,w_R-1,w_{R+1},\ldots,w_{t+1}).$
\par  \quad  b. If $r>R$, then $E(w)=(w_0-1,w_1-1,\ldots,w_{R-1}-1,w_R+1,w_{R+1},\ldots,w_{t+1}).$ \par \quad Return \textcircled{2}.
\par \textcircled{4} $R=-1$ implies that $w_n\leq1$ for all $0\leq n\leq t+1$, by (3.1) we get\par
\quad case 1: $E(w)=(1,1,\ldots,1,0)$. \par \quad case 2: $E(w)=(0,0,\ldots,0,1)$. Do
$E((0,0,\ldots,0,1))=(1,1,\ldots,1,0)$.\par (or: $R=-1$ implies that $w_n\leq1$ for all $0\leq n\leq t+1$, by
(3.2) we get\par \quad case 1: $E(w)=(0,1,1,\ldots,1,0)$. \par \quad case 2: $E(w)=(1,0,1\ldots,1,0)$. Do
$E((1,0,1\ldots,1,0))=(0,1,1\ldots,1,0)$.)\par Thus by finite steps, the desired result follows.
\end{proof}
For example, let $d$ be odd, $(e(d),e(2d),e(4d),e(8d),e(16d))=(2,4,1,0,0)\rightarrow (3,3,1,0,0)\rightarrow
(2,2,2,0,0)\rightarrow (1,1,1,1,0)$. Clearly, the algorithm in Lemma 3.2 is reversible. So reversing the algorithm
and noting that $S_k(\mathbb{P}_N)=-1$, by (3.3) we have the following theorem:
\begin{thm}
For any $P(x)\in OC(N)$, by finite $E-$transformations we can get $$E(\mathbb{P}_N(x))=P(x).$$ And furthermore
$$S_k(P)=-1+\sum\nolimits_{(t',d')\in G(E)} T_{t'}(k)(\pm C_{d'}(k)).$$
\end{thm}
In what follows we always suppose that $P(x)\in LC(N,i)$ and $E(\mathbb{P}_N(x))=P(x)$ and denote $S_k(P)$ by
$S_k$ again. What interests us are those sets $(t',d')\in G(E)$ such that $S_k\neq -1$ for some $1\leq k\leq N-2$.
Define $T(E):=\{t'|(t',d')\in G(E)\quad and \quad S_k\neq -1\text{ for some }1\leq k\leq N-2\}$. Obviously, for
any $t'\in T(E)$, we have $0\leq t'\leq t=v_2(N)$. Then $S_k$ can be written as
$$ S_k=-1+\sum\nolimits_{t'\in T(E)}T_{t'}(k)\sum\nolimits_{(t',d')\in G(E)}(\pm C_{d'}(k)).\leqno (3.4)$$
From (3.4), we can easily get the following results:
\begin{cor}
If $S_k\neq -1$, let $t'=v_2(k)$, then
\begin{enumerate}
\item[(a)]  $t'\in T(E)$.
\item[(b)]  $0\leq t'\leq t$.
\item[(c)]  $S_k=-1+2^{t'+1}\sum\nolimits_{(t',d')\in G(E)}\pm C_{d'}(k)$.
\item[(d)]  $2^{t'+1}|(S_k+1)$.
\end{enumerate}
\end{cor}
\begin{defn} (see \cite{5}) Let $S$ be a set contained in $\mathbb{N}$. We say that \(a\in S\) is a \textit{least-type divisor
in \(S\)}, if it can be deduced that \(c=a\) from \(c|a\) and \(c\in S\), that is, there is no other true divisor
of \(a\) in \(S\). Define $K(P)$ to be the set of least-type divisors in $\{k\mid 1\leq k\leq N-2 \; \text{and} \;
S_k\neq -1\}$.\end{defn}

\begin{cor}
If $k\in K(P)$ then $k|N$. In particular, we have $i\in K(P)$ and $i|N$.
\end{cor}
\begin{proof}
Suppose $k\in K(P)$. Let $p$ be an odd prime factor of $k$ and $d|M$. Since $v_2(k)\leq t=v_2(N)$ by Corollary 3.4
(b), it is sufficient to show that $v_p(k)\leq v_p(N)$. \par Case 1: $p|N$. Assume that $v_p(k)\geq v_p(N)+1$. It
follows that $v_p(k)\geq v_p(d)+1$. Suppose $k=p^{v_p(k)}k'$ and $d=p^{v_p(d)}d'$. Then by Lemma 2.1, we have
\begin{eqnarray*}
 C_d(k)  &  =  &  C_{p^{v_p(d)}}(k) C_{d'}(k) \\
         &  =  &  C_{p^{v_p(d)}}(p^{v_p(k)})C_{d'}(k')\\
         &  =  &  C_{p^{v_p(d)}}(p^{v_p(k)-1})C_{d'}(k')\\
         &  =  &   C_d(k/p).
\end{eqnarray*}
It yields that $S_{k/p}=S_k\neq -1$ by Lemma 3.1. This contradicts that $k\in K(P)$. So the assumption is not true
and hence we have $v_p(k)\leq v_p(N)$.
\par Case 2: $p\nmid N$. Clearly, $(d,p)=1$. Then by Lemma 2.1 (b) we have $C_d(k)=C_d(k/p)$. It yields that $S_{k/p}=S_k\neq
-1$ by Lemma 3.1. This contradicts that $k\in K(P)$. So this case does not exist.\par Thus we have $k|N$. Since
$i=min\{k\mid 1<k<N\, \text{and} \, S_k\neq -1\}$, it follows that $i\in K(P)$ and hence $i|N$.\end{proof}

\begin{lem}\textup{(In the proof of } \cite{1} \textup{ Theorem 3.3)}
Let $P(x)=\sum_{n=1}^{N-1}a_nx^n\in LC(N,i)$ and $1\leq j\leq i-1$. If $a_{li+j}=a_{li}$ for $0\leq l\leq m-1$
where $1\leq m\leq \frac{N-1}{i}-1$, then we have
\begin{eqnarray*}
(a)\qquad 0 & = & S_{mi+j}+1+(mi+j)(a_{mi+j}-a_{mi+j-1})\\
(b)\qquad 0 & = & (S_{(m+l)i+j}+1)+2((m+1)i+j)(a_{mi+j}-a_{mi+j-1})\\
  &   & +((m+1)i+j)(a_{(m+1)i+j}-a_{(m+1)i+j-1})
\end{eqnarray*}
\end{lem}
\begin{rem}
Borwein and Choi \cite{1} have pointed that $P(x)=\sum_{n=1}^{N-1}a_nx^n\in LC(N,i)$ is of the "periodicity" on
its coefficients, that is, $a_{li+j}=a_{li}$ for $0\leq l\leq \frac{N}{i}-1$ and $1\leq j\leq i-1$. We call it
"exterior periodicity", since $P(x)$ is also of the "interior periodicity" on $S_k$'s, that is, $S_{k}=-1$ for
$i\nmid k$. As seen from Lemma 3.7 above, the "periodicity" of $a_k$'s is totally determined by that of $S_k$'s.
Thus to prove the "periodicity" of $S_k$'s is key to prove Conjecture 1.2, and we will see it clearly from the
following theorem.
\end{rem}
\begin{thm}
Let $P(x)\in LC(N)$. $|K(P)|\leq1$ iff $P(x)$ is of the form \textup{(1.1)}.
\end{thm}
\begin{proof}
Since the sufficiency is easy to show by Lemma 3.7 (a), we only deal with its necessity. And the case of
$|K(P)|=0$ is trivial as it means that $P(x)=\mathbb{P}_N(x)$. \\Suppose $P(x)\in LC(N,i)$ We use induction on
$N$. $K(P)=\{i\}$ implies that $S_{k}=-1$ for $i\nmid k$. By Lemma 3.7 (a) we have
$a_{li}=a_{li+1}=\ldots=a_{li+i-1}$ for $0\leq l\leq \frac{N}{i}-1$. It follows that $P(x)=P_1(x)P_2(x^i)$ where
$P_1(x)=1+x+\cdots+x^{i-1}\in LC(i)$ and $P_2(x)\in LC(N/i)$. By induction, $P_1(x)$ and $P_2(x)$ are of the form
(1.1) and hence so is $P(x)$.
\end{proof}

\begin{lem}
Let $k=mi+j$ with $1\leq j\leq i-1$. \\
\textup{(a)}  If $v_2(k)\geq v_2(i)$, we have $v_2((m+1)i+j)\neq v_2(mi+j)$.\\
\textup{(b)}  If $v_2(k)=v_2(i)=t'$, we have $v_2((m+1)i+j)\geq t'+1$.
\end{lem}
\begin{proof}
(a) Assume that $v_2((m+1)i+j)=v_2(mi+j)=v_2(k):=t'$. Suppose that $(m+1)i+j=2^{t'}M_1$ and $mi+j=2^{t'}M_2$. By
substraction, we get $i=2^{t'}(M_1-M_2)$. Since both $M_1$ and $ M_2$ are odd, it follows that $v_2(i)\geq
2^{t'+1}>v_2(k)$. It is a contradiction.\\
(b)Noting that $v_2(k)=v_2(i)=t'$, we have $v_2(j)\geq t'$. \par Case 1: $2|m$. We claim $v_2(j)=t'$ since
otherwise $v_2(k)\geq t'+1$. It follows that $v_2((m+1)i+j)\geq t'+1$.\par Case 2: $2\nmid m$. We claim
$v_2(j)>t'$ since otherwise $v_2(k)\geq t'+1$. It follows that $v_2((m+1)i+j)\geq t'+1$ .\par Thus in either case
we have $v_2((m+1)i+j)\geq t'+1$.
\end{proof}

\begin{cor}
Let $P(x)\in LC(N)$. If $|T(E)|\leq1$, $P(x)$ is of the form \textup{(1.1)}.
\end{cor}
\begin{proof}
The case of $|T(E)|=0$ is trivial as it means that $P(x)=\mathbb{P}_N(x)$.\\ Suppose $P(x)\in LC(N,i)$ and
$t'=v_2(i)$. $|T(E)|=1$ implies $T(E)=\{t'\}$ and hence $S_k=-1$ for $v_2(k)\neq t'$. We claim that $K(P)=\{i\}$.
Otherwise assume $k'=min\{k\in K(P)\mid k>i\}$ then we have $i\nmid k', v_2(k')=t'$and $S_{k'}\neq -1$. Suppose
$k'=mi+j$ for $1\leq m\leq \frac{N-1}{i}-1$. By Lemma 3.10 (b) we have $v_2((m+1)i+j)\geq t'+1$ implying
$S_{((m+1)i+j)}=-1$. By Lemma 3.7 (b), we have
$$2(a_{mi+j}-a_{mi+j-1})+(a_{(m+1)i+j}-a_{(m+1)i+j-1})=0.$$
Since $a_{mi+j},a_{mi+j-1},a_{(m+1)i+j},a_{(m+1)i+j-1}=\pm1$, it follows that
$$a_{mi+j}=a_{mi+j-1}\qquad \text{and}\qquad a_{(m+1)i+j}=a_{(m+1)i+j-1}.$$
By Lemma 3.7 (a), we have $S_{k'}=S_{mi+j}=-1$. It is a contradiction. So our claim that $K(P)=\{i\}$ is true. The
result follows by Theorem 3.9.
\end{proof}

\begin{cor}
Let $P(x)\in LC(N)$. Conjecture 1.2 is true in the following cases:
\begin{enumerate}
\item[(a)]  $P(x)$ is square-free.
\item[(b)]  $N$ is odd.
\item[(c)]  $N=2^t$.
\item[(d)]  $N=2p^l$ where $p$ is an odd prime and $l\geq1$.
\item[(e)]  $N=2M$ with $M$ odd and $e(4d)=0$ for any $d|M$.
\end{enumerate}
\end{cor}
\begin{proof}
By Theorem 3.9 and Corollary 3.11, it is sufficient to prove that $|K(P)|\leq1$ or $|T(E)|\leq1$.
\par (a)Since $P(x)$ is square-free, by (3.1) and (3.2) we have that $\{0,t\}\supset T(E)$. If $t>0$ we claim that $0\notin T(E)$.
Otherwise by Corollary 3.4 (c) we have $S_1=-1-2C_2(1)=1$ which contradicts $S_1=-1$. So $|T(E)|\leq |\{t\}|=1$
for $t>0$. If $t=0$, it is clear that $|T(E)|\leq |\{0\}|=1$. \par (b) It is just the case of $t=0$ in (a).\par
(c) By Corollary 3.6, each number in $K(P)$ is the power of 2. Since $K(P)$ is the set of least-type divisors, we
have $K(P)=\{i\}$. \par (d) For any $k'\in K(P)$, Since $k'|N$ by Corollary 3.6, we have that $k'=p^{l'}$ or
$k'=2p^{l'}$ with $l'\leq l$.
\par case 1: $i=p^{l'}$. Assume $k_1=p^{l_1}\in K(P)$ (or: $k_2=2p^{l_2}\in K(P)$). Then we have
$k_1=p^{l_1}>i=p^{l'}$ (or: $k_2=2p^{l_2}>i=p^{l'}$). It follows that $l_1>l'$ (or: $2>p^{l'-l_2}$ also implying
$l_2>l'$). Thus $i|k_1$ (or: $i|k_2$). This contradicts that $K(P)$ is the set of least-type divisors. So we have
$K(P)=\{i\}$.\par case 2: $i=2p^{l_1}$. First we claim that $2p^{l_2}\notin K(P)$ for all $0\leq l_2\leq l$.
Otherwise it follows that $l_2>l_1$ implying $i|2p^{l_2}$. It is a contradiction. Then we claim that
$p^{l_2}\notin K(P)$ for all $0\leq l_2\leq l$. If not, assume $k'=min\{p^{l_2}\in K(P)\}$. Then we have that
$S_{i}=2i-1=4p^{l_1}-1, S_{k'}=\pm 2k'-1=\pm 2p^{l_2}-1$ by Lemma 3.7 (a). By Lemma 3.1 we have
\begin{eqnarray*}
S_{i}  &  =  &  \sum\nolimits_{j=1}^l(e(p^j)C_{p^j}(i)+e(2p^j)C_{2p^j}(i)+e(4p^j)C_{4p^j}(i))+e(1)C_1(i)+e(2)C_2(i)\\
       &  =  &  \sum\nolimits_{j=1}^l(e(p^j)+e(2p^j)C_{2}(2)+e(4p^j)C_{4}(2))C_{p^j}(p^{l_1})+e(1)C_1(2)+e(2)C_2(2)\\
       &  =  &  \sum\nolimits_{j=1}^l(e(p^j)+e(2p^j)-2e(4p^j))C_{p^j}(p^{l_1})+e(1)+e(2)\\
       &  =  &  \sum\nolimits_{j=1}^l(2-4e(4p^j))C_{p^j}(p^{l_1})+1\qquad \text{(  by Lemma 3.1)}\\
       &  =  &  \begin{cases} (2-4e(4p))(-1)+1 & l_1=0 \\
                             \sum\nolimits_{j=1}^{l_1}(2-4e(4p^j))(p^{j}-p^{j-1})+(2-4e(4p^{l_1+1}))(-p^{l_1})+1 &
                             l\geq1
                \end{cases}
\end{eqnarray*}
Similarly, noting that $2\nmid k'$ and $k'>i$ hence $l_2>l_1>0$, we have
\begin{eqnarray*}
S_{k'} &  =  &  \sum\nolimits_{j=1}^l(e(p^j)C_{p^j}(k')+e(2p^j)C_{2p^j}(k')+e(4p^j)C_{4p^j}(k'))+e(1)C_1(k)+e(2)C_2(k)\\
       &  =  &  \sum\nolimits_{j=1}^l(e(p^j)+e(2p^j)-2e(4p^j))C_{p^j}(p^{l_2})+e(1)-e(2)\\
       &  =  &  \sum\nolimits_{j=1}^{l_2}(2-4e(4p^j))(p^{j}-p^{j-1})+(2-4e(4p^{l_2+1}))(-p^{l_2})-1
\end{eqnarray*}
$e(1)-e(2)=1$ comes from that $S_{1}=(2-4e(4p))(-1)+e(1)-e(2)=-1$, $e(4p)=0,1$ and $e(1)-e(2)=\pm1$. We claim that
$l_2\geq2$. If not, since $l_2>0$, it follows that $l_2=1$. Therefore we have \setcounter{equation}{4}
\begin{eqnarray}
\pm2p-1=(2-4e(4p))(p-1)+(2-4e(4p))(-p)-1.
\end{eqnarray}
Dividing both sides of (3.5) by $p$ and noting that $e(4p)=0,1$, we get $2|p$. It is impossible. So we have
$l_2\geq2$.
\par Case 2.1: $l_1\geq 1$. Subtracting $S_{i}$ from $S_{k'}$, we have
\begin{eqnarray}
\pm 2p^{l_2}- 4p^{l_1} &  =  &  \sum\nolimits_{j={l_1+1}}^{l_2}(2-4e(4p^j))(p^{j}-p^{j-1})\nonumber\\
                       &     &  +(2-4e(4p^{l_2+1}))(-p^{l_2})-(2-4e(4p^{l_1+1}))(-p^{l_1})-2
\end{eqnarray}
\par Case 2.2: $l_1=0$. Then we have $4-1=(2-4e(4p))(-1)+1$. It follows that $e(4p)=1$. Substituting $e(4p)$ by 1 in $S_{k'}$, we have
\begin{eqnarray}
\pm 2p^{l_2}-1 &  =  &  -2(p-1)+\sum\nolimits_{j=2}^{l_2}(2-4e(4p^j))(p^{j}-p^{j-1})\nonumber\\
            &     &  +(2-4e(4p^{l_2+1}))(-p^{l_2})-1
\end{eqnarray}
Dividing both sides of (3.6) and (3.7) by $p$, we get the same result: $2|p$. It is impossible. So $p^{l_2}\notin
K(P)$ and hence $K(P)=\{i\}$.\par Thus we prove that $K(P)=\{i\}$ for $N=p^l$ with $p$ odd prime and $l\geq1$.
\par (e) For any $1<d|M$, since $e(4d)=0$, the $E-$transformation on $(e(d),e(2d),e(4d))$ must be
$(1,1,0)\rightarrow (2,0,0)$ or $(1,1,0)\rightarrow (0,2,0)$. Likewise for $d=1$, we have $(1,1)\rightarrow (0,2)$
or $(1,1)\rightarrow (2,0)$. It follows that $|T(E)|\leq|\{0\}|=1$.
\end{proof}

\section{Further Investigation and a Conjecture}
Let $P(x)\in LC(N,i)$. As seen from Remark 3.8 and Theorem 3.9, the key to prove that $P(x)$ is of the form (1.1)
is to prove that $K(P)=\{i\}$. We have proved in Corollary 3.11 that $|T(E)|=1$ implies $K(P)=\{i\}$. In fact, we
can apply induction on $|T(E)|$ under some constraints.
\par For convenience, define $\widetilde{LC}(N):=\bigcup_{2\leq
i|N}LC(N,i)$. And to express clearly, we use $E_{P_2 \rightarrow P_1}$ to denote the $E-$transformation such that
$E(P_2)=P_1$.
\begin{lem}
Assume that $P(x)\in \widetilde{LC}(N)$ is of the form \textup{(1.1)} if $|T(E_{\mathbb{P}_N \rightarrow P})|\leq
n$. Let $P_1(x)\in LC(N,i)$. If there exists $P_2(x)\in LC(N,i)$ satisfying \\ \textup{(1)} \; $|T(E_{\mathbb{P}_N
\rightarrow P_2})|\leq n;$ \\\textup{(2)} \; $T(E_{P_2 \rightarrow P_1})=\{t'\}\quad with\quad t'\geq v_2(i).$
\\Then $P_1(x)$ is of the form \textup{(1.1)} too.
\end{lem}
\begin{proof}
If $t' \in T(E_{\mathbb{P}_N \rightarrow P_2})$, then $|T(E_{\mathbb{P}_N \rightarrow P_1})|\leq n$. It follows
that $P_1(x)$ is of form (1.1) by assumption. In what follows suppose that $t'\notin T(E_{\mathbb{P}_N \rightarrow
P_2})$. We have
$$S_k(P_1)=S_k(P_2)+T_{t'}(k)\sum\nolimits_{(t',d')\in G(E_{P_2\rightarrow P_1})}(\pm C_{d'}(k)).\eqno (4.1)$$ We
claim that $K(P_1)=\{i\}$. Otherwise assume that $k'=min\{k\in K(P_1)\mid k>i\}$ then we have $i\nmid k'$ and
$S_{k'}(P_1)\neq -1$. We claim that $v_2(k')=t'$. Otherwise by (4.1) we have $S_{k'}(P_1)=S_{k'}(P_2)$. Since
$i\nmid k'$, it follows that $S_{k'}(P_2)=-1$ by Theorem 3.9 and hence $S_{k'}(P_1)=-1$. It is a contradiction. So
we have $v_2(k')=t'$. Suppose that $k'=mi+j$ with $1\leq j\leq i-1$. Since $v_2(k)\geq v_2(i)$, by Lemma 3.10 (a)
we have $v_2((m+1)i+j)\neq v_2(mi+j)=v_2(k')=t'$. It means that $v_2((m+1)i+j)\notin T(E_{P_2 \rightarrow P_1})$.
By (4.1), we have $S_{(m+1)i+j}(P_1)=S_{(m+1)i+j}(P_2)$. Since $i\nmid (m+1)i+j$, it follows that
$S_{(m+1)i+j}(P_2)=-1$ by Theorem 3.9 and hence  $S_{(m+1)i+j}(P_1)=-1$. By the similar discussion in Corollary
3.11, we have $S_{k'}(P_1)= -1$. It is a contradiction. So our claim that $K(P_1)=\{i\}$ is true. By Theorem 3.9,
$P_1(x)$ is of the form (1.1) too.
\end{proof}

\begin{lem}
If for any $P_1(x)\in LC(N,i)$ there exists $P_2(x)\in LC(N,i)$ satisfying
\\\textup{(1)}\; $|T(E_{\mathbb{P}_N \rightarrow P_1})|\leq|T(E_{\mathbb{P}_N \rightarrow P_2})|+1$;
\\\textup{(2)}\; $T(E_{P_2 \rightarrow P_1})=\{t'\}\quad with\quad t'\geq v_2(i).$ \\
Then Conjecture 1.2 is true for $N$.
\end{lem}
\begin{proof}
We use induction on $|T(E_{\mathbb{P}_N \rightarrow P})|$. We have proved in Corollary 3.11 that Conjecture 1.2 is
true for $|T(E_{\mathbb{P}_N\rightarrow P})|\leq 1$. Then assume Conjecture 1.2 is true for
$|T(E_{\mathbb{P}_N\rightarrow P})|\leq n$, i.e. $P(x)\in \widetilde{LC}(N)$ is of the form (1.1) if
$|T(E_{\mathbb{P}_N \rightarrow P})|\leq n$. By Lemma 4.1, $P(x)\in \widetilde{LC}(N)$ is of the form (1.1) if
$|T(E_{\mathbb{P}_N \rightarrow P})|\leq n+1$. That is, Conjecture 1.2 is also true for $|T(E_{\mathbb{P}_N
\rightarrow P})|\leq n+1$. It follows from Corollary 3.4 (b) that $|T(E_{\mathbb{P}_N \rightarrow P})|\leq
t=v_2(N)$. So by at most $t$ inductions on $|T(E_{\mathbb{P}_N \rightarrow P})|$, we can prove that Conjecture 1.2
is true for $|T(E_{\mathbb{P}_N \rightarrow P})|\leq t$, namely Conjecture 1.2 is true for $N$.
\end{proof}
\newpage
Take example for $N=2\times2\times3$. By computation, we find all eight polynomials in $\widetilde{LC}(N)$ as
follows:
\begin{eqnarray*}
P_1(x) & = & 1+x+x^2+x^3+x^4+x^5+x^6+x^7+x^8+x^9+x^{10}+x^{11}\\
       & = & \Phi_2(x)\Phi_4(x)\Phi_3(x)\Phi_6(x)\Phi_{12}(x);\\
P_2(x) & = & 1+x+x^2+x^3-x^4-x^5-x^6-x^7+x^8+x^9+x^{10}+x^{11}\\
       & = & \Phi_2(x)\Phi_4(x)\Phi_{24}(x);\\
P_3(x) & = & 1+x-x^2-x^3+x^4+x^5+x^6+x^7-x^8-x^9+x^{10}+x^{11}\\
       & = & \Phi_2(x)\Phi_4(x)\Phi_{12}^2(x);\\
P_4(x) & = & 1+x-x^2-x^3+x^4+x^5-x^6-x^7+x^8+x^9-x^{10}-x^{11}\\
       & = & -\Phi_1(x)\Phi_2^2(x)\Phi_3(x)\Phi_6(x)\Phi_{12}(x);\\
P_5(x) & = & 1+x-x^2-x^3-x^4-x^5+x^6+x^7+x^8+x^9-x^{10}-x^{11}\\
       & = & -\Phi_1(x)\Phi_2^2(x)\Phi_{24}(x);\\
P_6(x) & = & 1+x+x^2+x^3+x^4+x^5-x^6-x^7-x^8-x^9-x^{10}-x^{11}\\
       & = & -\Phi_1(x)\Phi_2^2(x)\Phi_3^2(x)\Phi_6^2(x);\\
P_7(x) & = & 1+x+x^2-x^3-x^4-x^5+x^6+x^7+x^8-x^9-x^{10}-x^{11}\\
       & = & -\Phi_1(x)\Phi_4(x)\Phi_3^2(x)\Phi_{12}(x);\\
P_8(x) & = & 1+x+x^2-x^3-x^4-x^5-x^6-x^7-x^8+x^9+x^{10}+x^{11}\\
       & = & \Phi_1^2(x)\Phi_2(x)\Phi_3^3(x)\Phi_6(x).
\end{eqnarray*}
The $E-$transformations of $P_i(x)$ and their $T(E)$'s are listed as follows:\\\\
\begin{tabular}{|c|c|c|l|c|}
\hline$P(x)$&  $i$& $(e(1),e(2),e(4))$ & $(e(3),e(6),e(12),e(24))$ & $T(E)$
\\\hline $P_1$& & & &$\varnothing$ \\\hline $P_2$ & 4& & $(1,1,1,0)\rightarrow(0,0,0,1)$&2\\\hline
$P_3$ & 2& &$(1,1,1,0)\rightarrow(0,0,2,0)$ &1\\\hline $P_4$ & 2&$(0,1,1)\rightarrow(1,2,0)$ & &1\\\hline $P_5$ &
2&$(0,1,1)\rightarrow(1,2,0)$ &$(1,1,1,0)\rightarrow(0,0,0,1)$ &1,2\\\hline$P_6$ &6&$(0,1,1)\rightarrow(1,2,0)$
&$(1,1,1,0)\rightarrow(2,2,0,0)$ &1\\\hline$P_7$ & 3&$(0,1,1)\rightarrow(1,0,1)$ &$(1,1,1,0)\rightarrow(2,0,1,0)$
&0\\\hline$P_8$ & 3&$(0,1,1)\rightarrow(2,1,0)$ &$(1,1,1,0)\rightarrow(2,0,1,0)\rightarrow(3,1,0,0)$ &0,1\\\hline
\end{tabular}
\\\\\\
By Corollary 3.11, $P_1,P_2,P_3,P_4,P_6,P_7$ are of the form (1.1). Since $T(E_{P_4 \rightarrow P_5})={2}$ and
$T(E_{P_7 \rightarrow P_8})={1}$, $P_5,P_8$ are also of the form (1.1) by Lemma 4.1. Thus Conjecture 1.2 is true
for $N=12$ by Lemma 4.2.\par

The difficulty is how to remove the "if" in Lemma 4.2. In other words we raise the following conjecture:
\begin{con}
For any $P_1(x)\in LC(N,i)$ there exists $P_2(x)\in LC(N,i)$ satisfying
\\\textup{(1)}\; $|T(E_{\mathbb{P}_N \rightarrow P_1})|\leq|T(E_{\mathbb{P}_N \rightarrow P_2})|+1$;
\\\textup{(2)}\; $T(E_{P_2 \rightarrow P_1})=\{t'\}\quad with\quad t'\geq v_2(i).$
\end{con}

% ----------------------------------------------------------------
%\bibliographystyle{amsplain}
%\bibliography{}

\end{document}